\documentclass{aimsx}
\usepackage{amsmath}
  \usepackage{paralist}
 \usepackage[colorlinks=true]{hyperref}
\hypersetup{urlcolor=blue, citecolor=red}
\usepackage{hyperref}
\usepackage{tikz,pgfplots}

\def\currentvolume{X}

    \def\ppages{X--XX}

\newtheorem{theorem}{Theorem}[section]
\newtheorem{corollary}[theorem]{Corollary}

\newtheorem{lemma}[theorem]{Lemma}

\theoremstyle{definition}
\newtheorem{definition}[theorem]{Definition}

\newcommand{\bbR}{\mathbb{R}}
\newcommand{\wL}{L}
\newcommand{\wS}{\widetilde{S}}
\newcommand{\cK}{\mathcal{K}}
\newcommand{\tff}{\tilde{f}}
\newcommand{\tii}{\tilde{i}}
\newcommand{\bii}{\bar{i}}
\newcommand{\hii}{\hat{i}}
\newcommand{\brii}{\breve{i}}
\newcommand{\tsfrac}  [2] { {\textstyle \frac{#1}{#2} } }
\newcommand{\nrm}  [1] {\Vert #1 \Vert}
\newcommand{\threebythree}[9]{ \begin{array}{ccc} #1 & #2 & #3 \\ #4 & #5 & #6 \\ #7 & #8 & #9 \end{array} }
\newcommand{\threebyone}[3]{ \begin{array}{c} #1 \\ #2 \\ #3 \end{array} }
\newcommand{\dd}{{\rm d}}
\newcommand{\ope}{{\rm e}}
\newcommand{\idop}{I\! d}

\title[Discrete gradient methods]
      {Discrete gradient methods for preserving a first integral of an ordinary differential equation}

\author[Richard A. Norton and G. R. W. Quispel]{}


 \email{richard.norton@latrobe.edu.au}
 \email{R.Quispel@latrobe.edu.au}

\thanks{The authors were supported by the Australian Research Council}

\begin{document}
\maketitle

\centerline{\scshape Richard A. Norton and G. R. W. Quispel}
\medskip
{\footnotesize
 \centerline{Department of Mathematics and Statistics}
   \centerline{La Trobe University}
   \centerline{Melbourne, Victoria 3086, Australia}
} 

%

\bigskip

\centerline{\emph{This paper is dedicated to Arieh Iserles, a dear friend and a wonderful colleague.}}

\begin{abstract}
In this paper we consider discrete gradient methods for approximating the solution and preserving a first integral (also called a constant of motion) of autonomous ordinary differential equations.  We prove under mild conditions for a large class of discrete gradient methods that the numerical solution exists and is locally unique, and that for arbitrary $p\in \mathbb{N}$ we may construct a method that is of order $p$.  In the proofs of these results we also show that the constants in the time step constraint and the error bounds may be chosen independently from the distance to critical points of the first integral.  

In the case when the first integral is quadratic, for arbitrary $p \in \mathbb{N}$, we have devised a new method that is linearly implicit at each time step and of order $p$. This new method has significant advantages in terms of efficiency.  We illustrate our theory with a numerical example.
\end{abstract}

\section{Introduction}
\label{sec intro}

Consider the autonomous ordinary differential equation (ODE)
\begin{equation}
\label{p1}
	\dot{x} = f(x) \qquad t > 0,
\end{equation}
where $x(t) \in \bbR^d$ for some $d \in \mathbb{N}$, $x(0) = x_0 \in \bbR^d$ is the initial condition and $f:\bbR^d \rightarrow \bbR^d$ is locally Lipschitz continuous.  Then given a bounded set $B \subset \bbR^d$, there exists a $T > 0$ such that for any $x_0 \in B$ the solution exists and remains bounded for $t \in [0,T]$ (see e.g. \cite[Thm. I.7.3 on p. 37]{SolvingODEs1}).  We assume that this ODE has a conserved first integral (also called a constant of motion) $I : \bbR^d \rightarrow \bbR$ such that 
\begin{equation}
\label{i1}
	I(x(t)) = I(x_0) \qquad \mbox{for all $t \in [0,T]$}.
\end{equation}
To simplify notation define $i(x) := \nabla I(x)$ for all $x\in \bbR^d$, and assume that $i(x)$ is locally Lipschitz continuous.  Define $\bbR_+ := \lbrace t \in \bbR: t > 0\rbrace$.  According to \cite{MQR99}, on $\lbrace x \in \bbR^d : i(x) \neq 0 \rbrace$ we may write \eqref{p1} as
\begin{equation}
\label{p3}
	\dot{x} = S(x) i(x)
\end{equation}
where $S(x) \in \bbR^{d \times d}$ is skew-symmetric ($S^T = -S$) and may be given by the so-called \emph{default} formula
\begin{equation}
\label{eqn3}
	S(x) = \frac{f(x) i(x)^T - i(x) f(x)^T}{|i(x)|^2}.
\end{equation}
In general, the choice of $S(x)$ satisfying $f(x) = S(x) i(x)$ is not unique.  Moreover, Proposition 2.1 in \cite{MQR99} states that if $f \in C^r(\bbR^d;\bbR^d)$ for $r \geq 1$ and $I$ is a Morse function (i.e. smooth with non-degenerate critical points) then $S$ in \eqref{eqn3} is $C^r$ and locally bounded on $\lbrace x \in \bbR^d : i(x) \neq 0 \rbrace$, and in the proof of Proposition 2.1 we also have that $|f(x)|/|i(x)|$ is locally bounded on $\lbrace x \in \bbR^d : i(x) \neq 0 \rbrace$.  In fact, the requirement that $f \in C^r(\bbR^d;\bbR^d)$ for $r \geq 1$ may be relaxed to $f$ locally Lipschitz continuous so that $S$ is also only locally Lipschitz continuous on $\lbrace x \in \bbR^d : i(x) \neq 0 \rbrace$.

Let us also make the assumption that $I$ is a Morse function so that for a bounded set $B \subset \bbR^d$ there exists a constant $C_1 = C_1(B)$ such that 
\begin{equation}
\label{eqn2}
	|f(x)| \leq C_1 |i(x)| \qquad \mbox{for all $x \in B$}.
\end{equation}
Note that from continuity it follows that if $i(x) = 0$ for $x \in B$, then $f(x) = 0$.  A useful constant throughout this paper will be $C_2 = C_2(B) := C_1 + \tsfrac{1}{5}$.  

Methods for approximating the solution to this type of ODE that simultaneously preserve the integral are of interest in many applications.  For example, Hamiltonian systems, Poisson systems, celestial mechanics, the Lotka-Volterra system and the undamped Duffing oscillator (see \cite{HLW} and references therein).  Here we consider \emph{discrete gradient methods} for approximating the solution to \eqref{p1} whilst exactly\footnote{i.e. up to round-off error or a larger specified tolerance.} preserving $I$ (see e.g.  \cite{MQR99,quispelcapel96,quispelturner96}).

Let us first define a special type of discretization of the gradient of $I$, a \emph{discrete gradient of $I$}.

\begin{definition}
(Gonzalez \cite{gonzalez96}).
A discrete gradient of $I$, denoted $\bii : \bbR^d \times \bbR^d \rightarrow \bbR^d$, is continuous and satisfies
$$
	\bii(x,x') \cdot (x' - x) = I(x') - I(x) \quad \mbox{and} \quad \bii(x,x) = i(x) \quad \mbox{for all $x,x' \in \bbR^d$.}
$$
\end{definition}

There are several ways of constructing a discrete gradient.  Two notable examples are the one used in the averaged vector field method (called the mean value discrete gradient in \cite{MQR99}, see also \cite{quispelmclaren08}) and the coordinate increment method \cite{abeitoh}.

Given a time step $h$ we define a discrete gradient method by the map $x \mapsto x'$\footnote{It will be our convention to let $x = x_n$ (the approximate solution at step $n$) and $x' = x_{n+1}$.}
\begin{equation}
\label{disc1}
	x' = \begin{cases}
		x + h \wS(x,x',h) \bii(x,x') & \mbox{if $i(x) \neq 0$},\\
		x & \mbox{if $i(x) = 0$},
		\end{cases}
\end{equation}
where $\bii$ is a discrete gradient of $I$ and $\wS$ is any skew-symmetric consistent approximation of $S$.  By consistent we mean that $\wS(x,x',h)$ is continuous and $\wS(x,x,0) = S(x)$ on $\lbrace x \in \bbR^d : i(x)\neq 0 \rbrace$.  All discrete gradient methods preserve $I$ because
\begin{equation}
\label{disc2}
	I(x') - I(x) = \bii(x,x') \cdot (x'-x) = h (\bii(x,x'))^T \wS(x,x',h) \bii(x,x') = 0.
\end{equation}
The final equality in \eqref{disc2} is because $\wS$ is skew-symmetric.

By discretizing the default formula for $S(x)$ given in \eqref{eqn3} we obtain an example of a discrete gradient method (there are many different possible discrete gradient methods for \eqref{p1}).  Let $\tii(x,x',h)$, $\hii(x,x',h)$ and $\brii(x,x',h)$ be consistent approximations to $i(x)$, so that they are all continuous and 
$$
	\tii(x,x,0) = \hii(x,x,0) = \brii(x,x,0) = i(x) \qquad \mbox{for all $x \in \bbR^d$,}
$$
let $\tilde{f}(x,x',h)$ be a consistent approximation of $f(x)$, and let $\bii(x,x')$ be a discrete gradient of $I$.  Then a discrete gradient method applied to \eqref{p1} is defined by \eqref{disc1} with $\wS(x,x',h)$ given by
\begin{equation}
\label{eqn4}
	\wS(x,x',h) =  \frac{\tff(x,x',h) (\tii(x,x',h))^T - \tii(x,x',h) (\tff(x,x',h))^T }{\hii(x,x',h) \cdot \brii(x,x',h) }.
\end{equation}

A useful way of describing the accuracy of a numerical method for solving \eqref{p1} is to determine its \emph{order of accuracy}.  For one-step methods this is defined by the truncation error around the point $x$ for a time step $h$.  

\begin{definition}
A one-step method $x \mapsto x'$ with time step $h$ for solving \eqref{p1} has order of accuracy $p \in \mathbb{N}$, if there exist positive constants $C$ and $H$ such that
$$
	|x' - x(t+h)| \leq C h^{p+1} \qquad \mbox{for all $h \in [0,H]$ and all $x \in B$,}
$$
where $x(\cdot)$ denotes the solution to \eqref{p1} with $x(t) = x$ for some $t \in \bbR_+$ and $B$ is a compact set in $\bbR^d$.  The constants $C$ and $H$ may depend on $B$ but should be independent of $x$ and $h$.  
\end{definition}

This definition (taken from \cite[Def. V.1.3]{gautschi}) is more precise about the dependencies for the constants $C$ and $H$ than the definitions for order $p$ given in other texts (e.g. \cite[Def. II.1.2]{SolvingODEs1} and \cite[Def. II.1.2]{HLW}) where it is defined by $|x' - x(t+h)| = \mathcal{O}(h^{p+1})$ as $h \rightarrow 0$.  These other definitions are ambiguous regarding how the hidden constant in $\mathcal{O}(\cdot)$ may depend on other parameters, and how small $h$ should be.   By using the definition from \cite{gautschi} in our results we can be sure that the constants in the definition of order $p$ do not depend on $|i(x)|$, which may be small.

Throughout this paper we will also make use of Banach's Fixed Point Theorem  (also called the Contraction Principle), see e.g. \cite[Thm. 3.1.2 on p. 74]{fixedpointtheory}.

\begin{theorem}[Banach's Fixed Point Theorem]
\label{thm banach}
Let $(X,d)$ be a non-empty complete metric space.  Let $T : X \rightarrow X$ be a contraction on $X$, i.e. there exists a $q \in (0,1)$ such that 
$$
	d(T(x),T(y)) \leq q d(x,y) \qquad \mbox{for all $x,y \in X$}.
$$
Then there exists a unique fixed point $x^* \in X$ such that $T(x^*) = x^*$.  Furthermore, the fixed point can be found by iteration, $x^{n+1} = T(x^n)$ for $n=0,1,2,\dotsc$ with $x^0 \in X$ arbitrary.
\end{theorem}

In Section \ref{sec exist} we prove that discrete gradient methods where $\wS$ has the form \eqref{eqn4} are well-defined in the sense that provided $h$ is sufficiently small and $\tff$, $\tii$, $\hii$, $\brii$ and $\bii$ satisfy certain consistency and local Lipschitz continuity conditions, then there exists a locally unique solution to \eqref{disc1}.  In Section \ref{sec error} we prove that for arbitrarily chosen $p \in \mathbb{N}$, if $\tff$, $\tii$, $\hii$, $\brii$ and $\bii$ satisfy two additional conditions ($\tff$ defines a method of order $p$ and $|\hii \cdot \brii - \tii \cdot \bii|$ is bounded in a special way) then we get a discrete gradient method of order $p$.

In Section \ref{sec modified} we consider discrete gradient methods from the perspective of doing computations.  Generally, each step of a discrete gradient method requires solving a nonlinear system of equations for $x'$ and this may add a significant amount to the computational cost of the method because an iterative scheme, such as Newton's method, must be employed at each step.  In the case when $I$ is quadratic we present a new method that is linearly implicit in $x'$ at each time step, so only a linear system of equations must be solved at each step.  We also show that Runge-Kutta methods, under very mild conditions on the coefficients, give an $\tff$ that satisfies the conditions required in Sections \ref{sec exist} and \ref{sec error}, and therefore we may use Runge-Kutta methods of order $p$ (for some $p \in \mathbb{N}$) to construct discrete gradient methods of order $p$.

Following the theory in these sections we present a numerical example in Section \ref{sec examples} to illustrate our theory, and finally, in Section \ref{sec conclusion} we discuss the implications of this work and possible avenues for future research.

\section{Existence and uniqueness}
\label{sec exist}

At each time step of the discrete gradient method we must (in general) solve a nonlinear system of equations (see \eqref{disc1}) for $x'$, but does the solution to this system of equations exist?  In this section we present a theorem that ensures for sufficiently small time step $h$, the map from $x \mapsto x'$ is well-defined in the sense that there exists a locally unique solution $x'$ to the system of equations \eqref{disc1} for the case when $\wS$ is given by \eqref{eqn4}.  

Usual techniques for achieving this type of result include applying the Implicit Function Theorem (see \cite{papi05}) or the Newton-Kantorovich Theorem (see \cite{ortega}).  For example, the Newton-Kantorovich Theorem is used in \cite{hairer00} to obtain existence of a numerical solution for a symmetric projection method, which requires solving a nonlinear system of equations at each time step.  In our experience these approaches for discrete gradient methods lead to a condition on the time step such as $h \leq C |i(x)|^r$ for some positive constants $C$ and $r$.  If we are close to a critical point of $I$ (i.e. when $|i(x)|$ is small) then the theory implies that we must also take $h$ small.  Our result and its proof below avoid this issue and we show that a solution to the nonlinear system of equations for a discrete gradient method (\eqref{disc1} with $\wS$ defined by \eqref{eqn4}) exists and is locally unique for a sufficiently small time step independent of $|i(x)|$ (and hence independent of the distance to critical points of $I$).

The local nature of our result (everything will depend on an initially chosen bounded set $B$) is only due to the local Lipschitz continuity of $f$ and $i$, and \eqref{eqn2}, rather than also depending on the distance to critical points of $I$.

We will require the following definition of a ball around a point $x \in \bbR^d$.  Given a constant $R > 0$ and $x \in \bbR^d$ define
$$
	B_R(x) := \lbrace z \in \bbR^d : |z-x| \leq \tsfrac{|i(x)|}{R} \rbrace.
$$
Note that if $i(x)=0$, then $B_R(x) = \lbrace x \rbrace$.

The following theorem ensures that, for sufficiently small $h$ and under certain local Lipschitz conditions, the map $x \mapsto x'$ defined by \eqref{disc1} and \eqref{eqn4} is locally well-defined, in the sense that there exists a locally unique solution to the nonlinear system of equations defined by \eqref{disc1} and \eqref{eqn4}.

\begin{theorem}
\label{thm exist}
Let $B$ be a bounded set in $\bbR^d$ and suppose there exist positive constants $R$, $\wL$ and $H$ such that for each $x \in B$ and all $u,v,w \in B_R(x)$ and $h \in [0,H)$, $\tilde{f}: \bbR^d \times \bbR^d \times \bbR_+ \rightarrow \bbR^d$ satisfies
\begin{equation}
\label{eqn11z}
\begin{split}
	\tff(x,x,0) &= f(x), \\
	|\tff(u,v,h) - \tff(w,v,h)| &\leq \wL |u-w|, \\
	|\tff(u,v,h) - \tff(u,w,h)| &\leq \wL |v - w|, \\
	|\tff(x,x,h) - \tff(x,x,0)| &\leq \wL h |i(x)|,
\end{split}
\end{equation}
$\bii : \bbR^d \times \bbR^d \rightarrow \bbR^d$ is a discrete gradient of $I$ satisfying 
\begin{equation}
\label{eqn11z1}
\begin{split}
	\bii(x,x) &= i(x), \\
	|\bii(u,v) - \bii(w,v)| &\leq \wL |u-w|, \\
	|\bii(u,v) - \bii(u,w)| &\leq \wL |v - w|, \\
\end{split}
\end{equation}
$\tii : \bbR^d \times \bbR^d \times \bbR_+ \rightarrow \bbR^d$ satisfies
\begin{equation}
\label{eqn12z}
\begin{split}
	\tii(x,x,0) &= i(x), \\
	|\tii(u,v,h) - \tii(w,v,h)| &\leq \wL |u-w|, \\
	|\tii(u,v,h) - \tii(u,w,h)| &\leq \wL |v - w|, \\
	|\tii(x,x,h) - \tii(x,x,0)| &\leq \wL h |i(x)|,
\end{split}
\end{equation}
and similarly for $\hii : \bbR^d \times \bbR^d \times \bbR_+ \rightarrow \bbR^d$ and $\brii : \bbR^d \times \bbR^d \times \bbR_+ \rightarrow \bbR^d$.  Let $C_2$ be the constant defined after \eqref{eqn2} and define 
$$
	R' := \max \lbrace R, 10 \wL \rbrace \qquad \mbox{ and } \qquad H' := \min \lbrace H, \tsfrac{1}{10\wL} , \tsfrac{1}{6C_2 R'},\tsfrac{1}{(36C_2 + 6) \wL}\rbrace.
$$
Then for each $x \in B$ and $h \in [0,H')$ there exists a unique $x' \in B_{R'}(x)$ satisfying \eqref{disc1} where $\wS$ is given by the formula \eqref{eqn4}.  
\end{theorem}

\begin{proof}
Note that if $x \in B$ and $i(x) = 0$ then $x' = x$ is the unique solution to \eqref{disc1} in $B_{R'}(x) = \lbrace x \rbrace$.  For the case when $i(x) \neq 0$ we will apply Banach's Fixed Point Theorem (Theorem \ref{thm banach}) to prove our result.  Let $R'$ and $H'$ be defined as in the theorem and for fixed $x \in B$, such that $i(x) \neq 0$, define $X := B_{R'}(x)$.  $X$ is a closed subset of $\bbR^d$, so together with the metric $| \cdot |$, it is a complete metric space.  Also fix $h \in [0, H')$, and define $T : X \rightarrow \bbR^d$ by
$$
	T(z) := x + h \wS(x,z,h) \bii(x,z) \qquad \mbox{for all $z \in X$,}
$$	
where $\wS$ is given by \eqref{eqn4}.  To satisfy the assumptions of Theorem \ref{thm banach} we must show that $T(z) \in X$ for all $z \in X$ and that $T$ is a contraction on $(X,|\cdot|)$.

Let $z \in X$.  Using \eqref{eqn12z}, $z \in B_{R'}(x) \subset B_R(x)$ (since $R' \geq R$), $R' \geq 10L$ and $h \leq \tsfrac{1}{10L}$ we have
\begin{equation}
\label{eqn12z1}
\begin{split}
	|\tii(x,z,h)| 
	&\leq |i(x)| + |\tii(x,z,h) - \tii(x,x,h)| + |\tii(x,x,h) - \tii(x,x,0)| \\
	&\leq (1 + \tsfrac{\wL}{R'} + \wL h) |i(x)| 
	\leq \tsfrac{6}{5} |i(x)|.
\end{split}
\end{equation}
We can derive similar inequalities for $\hii$ and $\brii$, and for $\bii$ we can derive
\begin{equation}
\label{eqn12z2}
	|\bii(x,z)| \leq \tsfrac{11}{10} |i(x)|.
\end{equation}
Using \eqref{eqn11z}, \eqref{eqn2}, $C_2 := C_1 + \tsfrac{1}{5}$, $R' \geq 10L$ and $h \leq \tsfrac{1}{10L}$ we also get
\begin{equation}
\label{eqn12z3}
	|\tff(x,z,h)| \leq C_2 |i(x)|.
\end{equation}
Using \eqref{eqn12z} and \eqref{eqn12z1} for $\hii$ and $\brii$, $z \in B_{R'}(x)$, $R' \geq 10L$, $h \leq \tsfrac{1}{10L}$ and writing $i$ instead of $i(x)$ we get
\begin{equation}
\label{eqn17z}
\begin{split}
\hii(x,z,h) \cdot \brii(x,z,h)
	=&|i|^2 + [ (\hii(x,z,h) - \hii(x,x,h)) \\
	&+ (\hii(x,x,h) - \hii(x,x,0))]  \cdot \brii(x,z,h) \\
	& + i \cdot [ (\brii(x,z,h) - \brii(x,x,h)) + (\brii(x,x,h) - \brii(x,x,0))] \\
	\geq& |i|^2 -  (\tsfrac{\wL|i|}{R'} + \wL h |i|) \tsfrac{6}{5} |i|
	- |i| (\tsfrac{\wL|i|}{R'} + \wL h |i|)\\
	\geq& |i|^2 -  \tsfrac{2}{10} \cdot \tsfrac{6}{5} |i|^2 - \tsfrac{2}{10} |i|^2
	= \tsfrac{28}{50} |i|^2 
	> \tsfrac{1}{2} |i|^2.
\end{split}
\end{equation}
We get $T(z) \in X$ from the following inequality, where we have used \eqref{eqn17z}, \eqref{eqn12z3}, \eqref{eqn12z1}, \eqref{eqn12z2} and $h \leq H'$ to get
\begin{align*}
	|T(z) - x| &= h |\wS(x,z,h) \bii(x,z)| 
	= h \left| \tsfrac{ (\tii \cdot \bii) \tff - (\tff \cdot \bii) \tii}{\hii \cdot \brii} \right| \\
	&\leq \tsfrac{4 h}{|i(x)|^2} |\tff(x,z,h)| |\tii(x,z,h)| |\bii(x,z)| \\
	&\leq 4 h C_2 \tsfrac{6}{5} \cdot \tsfrac{11}{10} |i(x)| 
	= \tsfrac{132}{25} C_2 h |i(x)|
	\leq 6 C_2 h |i(x)| 
	\leq \tsfrac{|i(x)|}{R'}.
\end{align*}

To show $T$ is a contraction, let $z,z'\in X$.  Using \eqref{eqn17z}, \eqref{eqn12z} and \eqref{eqn12z1} for $\hii$ and $\brii$, and writing $\hii(x,z,h)$ as $\hii(z)$ etc. we get
\begin{equation}
\label{eqn13z1}
\begin{split}
	\left| \tsfrac{1}{\hii(z)\cdot \brii(z)}  \!-\! \tsfrac{1}{\hii(z')\cdot\brii(z')} \right|
	& = \left|\tsfrac{ \hii(z')\cdot\brii(z') - \hii(z)\cdot\brii(z) }{(\hii(z')\cdot\brii(z')) ( \hii(z)\cdot\brii(z) )} \right| \\
	& \leq \tsfrac{4}{|i(x)|^4} | \hii(z')\cdot\brii(z') - \hii(z)\cdot\brii(z) | \\
	& \leq \tsfrac{4}{|i(x)|^4} \left( | (\hii(z') - \hii(z)) \cdot\brii(z')| + |\hii(z)\cdot(\brii(z') - \brii(z)) | \right) \\
	& \leq \tsfrac{8}{|i(x)|^4} \wL |z-z'|  \tsfrac{6}{5} |i(x)| 
	= \tsfrac{48 \wL}{5 |i(x)|^3} |z-z'| 
	\leq \tsfrac{10 \wL}{|i(x)|^3}|z-z'|. 
\end{split}
\end{equation}
Using \eqref{eqn11z}, \eqref{eqn11z1}, \eqref{eqn12z}, \eqref{eqn12z1}, \eqref{eqn12z2}, \eqref{eqn12z3}, \eqref{eqn17z} and \eqref{eqn13z1} we get 
\begin{equation}
\label{eqn13z2}
\begin{split}
	\Bigl| \tsfrac{\tii(z) \cdot \bii(z)}{\hii(z)\cdot\brii(z)} {\scriptstyle \tff(z) -} \tsfrac{\tii(z') \cdot \bii(z')}{\hii(z')\cdot\brii(z')} {\scriptstyle \tff(z')} \Bigr| 
	 \leq& \left| \tsfrac{[\tii(z) - \tii(z')] \cdot \bii(z) \tff(z)}{\hii(z)\cdot\brii(z)} \right| 
	 + \left| \tsfrac{\tii(z') \cdot [\bii(z) - \bii(z')] \tff(z)}{\hii(z)\cdot\brii(z)} \right| \\
	 &+ \left| \tsfrac{\tii(z')\cdot\bii(z') [\tff(z) - \tff(z')] }{\hii(z)\cdot\brii(z)} \right| \\
	& + \left| \left( \tsfrac{1}{\hii(z)\cdot \brii(z)} - \tsfrac{1}{\hii(z')\cdot\brii(z')} \right) {\scriptstyle \tii(z')\cdot\bii(z') \tff(z') } \right| \\
	\leq& {\scriptstyle \frac{2}{|i(x)|^2}} \bigl( {\scriptstyle \wL |z-z'| \frac{11}{10} |i(x)| C_2 |i(x)|  + \frac{6}{5} |i(x)| \wL |z-z'| C_2 |i(x)| }\\
	& {\scriptstyle+ \frac{6}{5}|i(x)| \frac{11}{10} |i(x)| \wL|z-z'|} \bigr) 
	{\scriptstyle + \frac{10 \wL}{|i(x)|^3} |z-z'| \frac{6}{5} |i(x)| \frac{11}{10} |i(x)| C_2|i(x)|} \\
	=& \left( \tsfrac{89}{5}C_2 + \tsfrac{66}{25} \right) \wL |z-z'|  \\
	\leq& \left( 18C_2 + 3 \right)\wL |z-z'|.
\end{split}
\end{equation}
Using a similar argument we can also derive
\begin{equation}
\label{eqn13z3}
	\Bigl| \tsfrac{\tff(z) \cdot \bii(z)}{\hii(z)\cdot\brii(z)} \tii(z) - \tsfrac{\tff(z') \cdot \bii(z')}{\hii(z')\cdot\brii(z')} \tii(z') \Bigr| 
	\leq 	\left( 18C_2 + 3 \right)\wL |z-z'|.
\end{equation}
Now using \eqref{eqn13z2}, \eqref{eqn13z3} and $h < H'$ we get
\begin{align*}
	|T(z) - T(z')| &\leq h \Bigl| \tsfrac{\tii(z)\cdot\bii(z) \tff(z) - \tff(z)\cdot\bii(z) \tii(z)}{\hii(z)\cdot\brii(z)} - \tsfrac{\tii(z')\cdot\bii(z') \tff(z') - \tff(z')\cdot\bii(z') \tii(z')}{\hii(z')\cdot\brii(z')} \Bigr| \\
	&\leq h \Bigl| \tsfrac{\tii(z) \cdot \bii(z)}{\hii(z)\cdot\brii(z)} \tff(z) - \tsfrac{\tii(z') \cdot \bii(z')}{\hii(z')\cdot\brii(z')} \tff(z') \Bigr|  
	+ h \Bigl| \tsfrac{\tff(z) \cdot \bii(z)}{\hii(z)\cdot\brii(z)} \tii(z) - \tsfrac{\tff(z') \cdot \bii(z')}{\hii(z')\cdot\brii(z')} \tii(z') \Bigr| \\
	&\leq (36 C_2 + 6) \wL h |z-z'|,
\end{align*}
where $(36C_2 + 6) \wL h < 1$.  
Therefore, $T$ is a contraction on $(X,|\cdot|)$ and by Theorem \ref{thm banach} there exists a unique $x' \in X$ such that $T(x') = x'$.   By the definition of $T$ it follows that $x'$ satisfies \eqref{disc1} where $\wS$ is given by \eqref{eqn4}.  
\end{proof}

\section{Order of accuracy}
\label{sec error}

In this section we give sufficient conditions for a discrete gradient method defined by \eqref{disc1} and \eqref{eqn4} to be of order $p$ for arbitrarily chosen $p \in \mathbb{N}$.  In addition to requiring the same conditions as in Theorem \ref{thm exist}, we also require two further conditions.

The following two lemmas will be used to prove our main result, Theorem \ref{thm error2}.

\begin{lemma}
\label{lem exist f}
For a bounded set $B \subset \bbR^d$, let $R,L,H,\tff,C_2,R'$ and $H'$ be defined as in Theorem \ref{thm exist}.  Then, for each fixed $x \in B$ and $h \in [0,H')$ there exists a unique $y \in B_{6R'}(x)$ such that 
\begin{equation}
\label{eqn21z}
	y = x + h \tff(x,y,h).
\end{equation}
\end{lemma}

\begin{proof}
Fix $x \in B$ and $h \in [0,H')$, and define $X := B_{6R'}(x)$ and $T(z):= x + h \tff(x,z,h)$ for each $z \in X$.  For $z \in X$ use \eqref{eqn12z3} and $h \leq \tsfrac{1}{6 C_2 R'}$ to get
$$
	|T(z) - x| = h |\tff(x,z,h)| \leq h C_2  |i(x)| \leq \tsfrac{|i(x)|}{6R'},
$$
so $T(z) \in X$.  For $z' \in X$, use \eqref{eqn11z} and $h \leq \tsfrac{1}{10L}$ to get
$$
	|T(z)-T(z')| = h |\tff(x,z,h) - \tff(x,z',h)| \leq L h |z-z'| \leq \tsfrac{1}{10}|z-z'|.
$$
Hence $T:X \rightarrow X$ is a contraction and the result follows by applying Theorem \ref{thm banach}.  For the case when $i(x)=0$ note that $\tff(x,z,h) = 0$ for $z \in X = \lbrace x \rbrace$ (use \eqref{eqn11z} and $|i(x)| = 0$).
\end{proof}

\begin{lemma}
\label{lem 1}
In addition to the hypotheses of Theorem \ref{thm exist} suppose that for each $x \in B$ and all $u,v \in B_{5R'}(x)$, $f$ satisfies
$$
	|f(u) - f(v)| \leq L |u-v|.
$$
Let $x(\cdot)$ denote the exact solution to \eqref{p1} with $x(t) = x$ for some $t \in \bbR_+$ and $h \in [0,H')$, then $x(s)$ exists and satisfies $x(s) \in B_{5R'}(x)$ for all $s \in [t,t+h]$.
\end{lemma}

\begin{proof}
If $i(x) = 0$ then (by \eqref{eqn2}) we are at a stationary point and the result is trivial.  Suppose $i(x) \neq 0$.  Existence theory for ODEs (see e.g. \cite[Thm. I.7.3 on p. 37]{SolvingODEs1}) implies there exists a $T>t$ such that $x(s) \in B_{5R'}(x)$ for all $s \in [t,T]$.  If $T\geq t + h$ then we are done, so suppose $T <t +h$.  Existence theory also implies that the solution $x(s)$ exists for $s \in [T, T']$ for some $T' > T$, even though it may not be in $B_{5R'}(x)$.

For each $s \in [t,T]$ we have 
$$
	|f(x(s))| \leq |f(x)| + L|x(s)-x(t)| \leq |f(x)| + \int_t^s |f(x(r))| \dd r,
$$
so by the Gronwall Inequality (see e.g. \cite[Thm. 1.1 on p. 24]{hartman}), \eqref{eqn2}, $h < H'$ and $i(x) \neq 0$ we have
$$
	|f(x(s))| \leq |f(x)| \ope^{L h} < C_1 |i(x)| \ope^{1/10} < \tsfrac{6 C_1 |i(x)|}{5}.
$$	
Therefore, for each $s \in [t,T]$,
$$
	|x(s) - x| \leq \int_t^s |f(x(r))| \dd r \leq \tsfrac{6 C_1h |i(x)|}{5} < \tsfrac{6 C_2h |i(x)|}{5} < \tsfrac{|i(x)|}{5 R'}.
$$
Since this inequality is strict and $x(\cdot)$ exists up to $T'$ and is continuous, there exists an $\epsilon > 0$ such that $x(r) \in B_{5R'}(x)$ for all $r \in [t,s+\epsilon]$. 

To complete the proof let us argue by contradiction.  Suppose there exists a $s \in [t, t+h]$ such that $x(s) \notin B_{5R'}(x)$ (this includes the case when $x(t+h)$ does not exist).  Then by continuity of $x(\cdot)$, $x(t) = x$ and since $B_{5R'}(x)$ is closed, there exists a $s' \in [t, s)$ and a $\delta > 0$ such that $x(r) \in B_{5R'}(x)$ for all $r \in [t,s']$ and $x(r) \notin B_{5R'}(x)$ for all $r \in (s',s'+\delta)$.  However, by the above argument there exists an $\epsilon > 0$ such that $x(r) \in B_{5R'}(x)$ for all $r \in [t,s'+\epsilon)$.  A contradiction.  Therefore, $x(s) \in B_{5R'}(x)$ for all $s \in [t,t+h]$.
\end{proof}

The extra Lipschitz continuity condition on $f$ in Lemma \ref{lem 1} follows from our earlier assumption that $f$ is locally Lipschitz.

Now let us present the main theorem of this section, where we show that under certain conditions the discrete gradient method defined by \eqref{disc1} and \eqref{eqn4} is of order $p$, for some $p \in \mathbb{N}$.

\begin{theorem}
\label{thm error2}
For a compact set $B \subset \bbR^d$, let $R$, $L$, $H$, $\tff$, $\bii$, $\tii$, $\hii$ $\brii$, $C_2$, $R'$ and $H'$ be defined as in Theorem \ref{thm exist} and let $f$ satisfy the Lipschitz condition in Lemma \ref{lem 1}.  

For each $x \in B$ and $h \in [0,H')$ 
\begin{enumerate}
\item let $x' \in B_{R'}(x)$ be the unique solution to \eqref{disc1} with $\wS$ defined by \eqref{eqn4} (which exists by Theorem \ref{thm exist}), 
\item let $y \in B_{6R'}(x)$ be the unique solution to \eqref{eqn21z} (which exists by Lemma \ref{lem exist f}), and 
\item let $x(\cdot)$ denote the exact solution to \eqref{p1} satisfying $x(t) = x$ for some $t \in \bbR_+$ (which exists on $[t,t+h]$ by Lemma \ref{lem 1}).
\end{enumerate}
Also suppose that 
\begin{enumerate}
\setcounter{enumi}{3}
\item $\tff$ is such that the method defined by \eqref{eqn21z} is of order $p$ for some $p \in \mathbb{N}$, i.e. there exist positive constants $C_3$ and $H_3 < H'$ such that
\begin{equation}
\label{eqn30z}
	|y - x(t+h)| \leq C_3 h^{p+1}, \quad \mbox{for all $h \in [0,H_3]$ and all $x \in B$,}
\end{equation}
and
\item there exists a positive constant $C_4$ such that for each $x \in B$ and all $h \in [0,H_3]$,
\begin{equation}
\label{eqn30z1}
	|\hii(x,x',h)\cdot \brii(x,x',h) - \tii(x,x',h) \cdot \bii(x,x')| 
	\leq C_4
	\left( |x' - x(t+h)| + h^{p+1} \right) |i(x)|.
\end{equation}
\end{enumerate}
Then the discrete gradient method defined by \eqref{disc1} with $\wS$ given by \eqref{eqn4} is also of order $p$, so that there exist positive constants $C_5$ and $H_5$  such that 
$$
	|x' - x(t+h)| \leq C_5 h^{p+1}, \quad \mbox{for all $h \in [0,H_5]$ and all $x \in B$.}
$$
\end{theorem}

\begin{proof}
Define $0 < H_5 \leq \min \lbrace H_3, \tsfrac{1}{30 C_2 C_4} \rbrace < H'$ and $C_5 = \tsfrac{5 C_4}{24L} + \tsfrac{51 C_3}{4}$.  Fix $x \in B$ and $h \in [0,H_5]$.  The first step in the proof is to bound $|h \tff(x,x',h) \cdot \bii(x,x')|$.  We get
\begin{equation}
\label{eqn23z}
\begin{split}
	|h \tff(x,x',h) \cdot \bii(x,x') |  
	\leq& | h \tff(x,y,h) \cdot \bii(x,x(t+h))| \\
	& + | h \tff(x,y,h)\cdot [ \bii(x,x') - \bii(x,x(t+h)]|  \\
	& + | h [\tff(x,x',h) - \tff(x,y,h)] \cdot \bii(x,x') | \\
	=:& T_1 + T_2 + T_3.
\end{split}
\end{equation}
Now bound each $T_i$ separately.  Using $(x(t+h)-x)\cdot \bii(x,x(t+h))= I(x(t+h)) - I(x) = 0$, \eqref{eqn30z} and \eqref{eqn12z2} (with $z = x(t+h) \in B_{5R'}(x) \subset B_{R'}(x)$) we get
\begin{equation}
\label{eqn21az}
\begin{split}
	T_1 
	:=& |h \tff(x,y,h) \cdot \bii(x,x(t+h))| 
	 = | (y-x) \cdot \bii(x,x(t+h))| \\
	=& | (y - x(t+h)) \cdot \bii(x,x(t+h))| \\
	\leq& \tsfrac{11}{10} C_3 h^{p+1} |i(x)|.
\end{split}
\end{equation}
Using \eqref{eqn12z3} (with $z = y \in B_{6R'}(x) \subset B_{R'}(x)$), \eqref{eqn11z1} and $h \leq H' \leq \tsfrac{1}{36 C_2 L}$ we get
\begin{equation}
\label{eqn31z}
\begin{split}
	T_2 :=&  | h \tff(x,y,h)\cdot [ \bii(x,x') - \bii(x,x(t+h)]| \\
	\leq& h C_2 |i(x)| L |x' - x(t+h)| \\
	\leq& \tsfrac{1}{36} |i(x)| |x' - x(t+h)|.
\end{split}
\end{equation}
And using \eqref{eqn11z}, \eqref{eqn12z2} (with $z = x' \in B_{R'}(x)$) and $h \leq H' \leq \tsfrac{1}{10L}$ we get
\begin{equation}
\label{eqn32z}
\begin{split}
	T_3 &:= | h [\tff(x,x',h) - \tff(x,y,h)] \cdot \bii(x,x') | \\
	&\leq h L |x'-y| \tsfrac{11}{10} |i(x)| \\
	&\leq \tsfrac{11}{100} |i(x)| |x'-y|.
\end{split}
\end{equation}
Putting \eqref{eqn23z} together with \eqref{eqn21az}, \eqref{eqn31z} and \eqref{eqn32z} we get
\begin{equation}
\label{eqn33z}
|h \tff(x,x',h) \cdot \bii(x,x') | \leq \left( \tsfrac{11}{10} C_3 h^{p+1} + \tsfrac{1}{36} |x' - x(t+h)| + \tsfrac{11}{100} |x'-y| \right) |i(x)|.
\end{equation}
Now consider bounding $|x' - y|$.  We get
\begin{equation}
\label{eqn34z}
\begin{split}
	|x'-y| &= {\scriptstyle \left| h \tff(x,x',h) \frac{\tii(x,x',h) \cdot \bii(x,x')}{\hii(x,x',h) \cdot \brii(x,x',h)} 
	- h \tii(x,x',h) \frac{\tff(x,x',h) \cdot \bii(x,x')}{\hii(x,x',h) \cdot \brii(x,x',h)} 
	- h \tff(x,y,h) \right| } \\
	&\leq {\scriptstyle h \left| \tff(x,x',h) \tsfrac{\tii(x,x',h) \cdot \bii(x,x')}{\hii(x,x',h) \cdot \brii(x,x',h)} 
	-  \tff(x,y,h) \right| 
	 + \left| \tii(x,x',h) \tsfrac{h \tff(x,x',h) \cdot \bii(x,x')}{\hii(x,x',h) \cdot \brii(x,x',h)} \right| }\\
	 &=: T_4 + T_5.
\end{split}
\end{equation}
Now bound $T_4$ and $T_5$ separately.  Using \eqref{eqn11z}, \eqref{eqn12z3}, \eqref{eqn17z}, \eqref{eqn12z1}, \eqref{eqn12z2} (with $z = x' \in B_{R'}(x)$ and $z = y \in B_{R'}(x)$), \eqref{eqn30z1} and $h\leq \min\lbrace H_3, \tsfrac{1}{10L}, \tsfrac{1}{36 C_2 L} \rbrace$ we get
\begin{equation}
\label{eqn35z}
\begin{split}
	T_4 &:= h \left| \tff(x,x',h) \tsfrac{\tii(x,x',h) \cdot \bii(x,x')}{\hii(x,x',h) \cdot \brii(x,x',h)} 
	-  \tff(x,y,h) \right| \\
	&\leq {\scriptstyle h \left| (\tff(x,x',h) - \tff(x,y,h)) 
	\frac{\tii(x,x',h) \cdot \bii(x,x')}{\hii(x,x',h) \cdot \brii(x,x',h)} \right|
	+ h \left| \tff(x,y,h) \left( \frac{\tii(x,x',h) \cdot \bii(x,x')}{\hii(x,x',h) \cdot \brii(x,x',h)} - 1  \right)  \right| }\\
	&\leq h L |x'-y| \left| \tsfrac{\tii(x,x',h) \cdot \bii(x,x')}{\hii(x,x',h) \cdot \brii(x,x',h)} \right|
	+ h C_2 |i(x)| \left| \tsfrac{\tii(x,x',h) \cdot \bii(x,x')}{\hii(x,x',h) \cdot \brii(x,x',h)} - 1 \right|  \\
	&\leq 2 h L |x'-y| \tsfrac{6}{5}\cdot \tsfrac{11}{10} 
	+ 2 C_2 C_4 h (|x'-x(t+h)| + h^{p+1})  \\
	&\leq \tsfrac{3}{10} |x'-y| + 2C_2 C_4 h  |x' - x(t+h)| + \tsfrac{C_4}{18L} h^{p+1}. 
\end{split}
\end{equation}
And using \eqref{eqn17z} (with $z = x'$), \eqref{eqn12z1} and \eqref{eqn33z} we get
\begin{equation}
\label{eqn37z}
\begin{split}
	T_5 &:=  \left| \tii(x,x',h) \tsfrac{h \tff(x,x',h) \cdot \bii(x,x')}{\hii(x,x',h) \cdot \brii(x,x',h)} \right| \\
	&\leq \tsfrac{12}{5|i(x)|} |h \tff(x,x',h) \cdot \bii(x,x') | \\
	&\leq \tsfrac{132}{50} C_3 h^{p+1} + \tsfrac{1}{15} |x' - x(t+h)| + \tsfrac{132}{500} |x' - y| \\
	&\leq 3 C_3 h^{p+1} + \tsfrac{1}{15} |x'-x(t+h)| + \tsfrac{3}{10} |x'-y|.
\end{split}
\end{equation}
Putting \eqref{eqn34z} together with \eqref{eqn35z} and \eqref{eqn37z} we get
$$
	|x'-y| \leq \tsfrac{3}{5} |x'-y| + \left( 2 C_2 C_3 h + \tsfrac{1}{15} \right) |x' + x(t+h)| + \left( \tsfrac{C_4}{18L} + 3C_3 \right) h^{p+1},
$$
and hence
\begin{equation}
\label{eqn38z}
	|x'-y| \leq \left( 5 C_2 C_4 h + \tsfrac{1}{6} \right) |x' + x(t+h)| + \left( \tsfrac{5 C_4}{36L} + \tsfrac{15C_3}{2} \right) h^{p+1}.
\end{equation}
Our result now follows easily from \eqref{eqn38z} and \eqref{eqn30z} since $h \leq \min \lbrace H_3, \tsfrac{1}{30 C_2 C_4} \rbrace$:
\begin{align*}
	|x'-x(t+h)| 
	&\leq |x' - y| + |y - x(t+h)| \\
	&\leq \left( 5 C_2 C_4 h + \tsfrac{1}{6} \right) |x' + x(t+h)| + \left( \tsfrac{5 C_4}{36 L} + \tsfrac{17C_3}{2} \right) h^{p+1} \\
	&\leq \tsfrac{ 1}{3} |x'-x(t+h)| + \left( \tsfrac{5 C_4}{36 L} + \tsfrac{17C_3}{2} \right) h^{p+1}.
\end{align*}
Therefore,
$$
	|x'-x(t+h)| \leq \tsfrac{3}{2} \left( \tsfrac{5 C_4}{36 L} + \tsfrac{17C_3}{2} \right) h^{p+1} = \left( \tsfrac{5 C_4}{24L} + \tsfrac{51 C_3}{4} \right) h^{p+1} = C_5 h^{p+1}.
$$

\end{proof}

\section{Discrete gradient methods for computation}
\label{sec modified}

In this section we consider discrete gradient methods that may be used in computations, and how we might choose $\tff$, $\tii$, $\bii$, $\hii$ and $\brii$ so that they satisfy the hypotheses of Theorems \ref{thm exist} and \ref{thm error2}.  We also consider how to make the nonlinear system of equations defined by \eqref{disc1} and \eqref{eqn4} as easy as possible to solve at each time step.  In the case when $I$ is quadratic we achieve this by constructing a discrete gradient method so that it is linearly implicit in $x'$.  In general this is not possible, except in the case when $I$ is quadratic.

Since $f(x)$ and $i(x)$ are locally Lipschitz continuous, for a given bounded set $B \subset \bbR^d$ and constant $R_0 > 0$ there exists a constant $L_0>0$ such that for all $x \in B$ and for all $u,v \in B_{R_0}(x)$ 
\begin{equation}
\label{eqn40z}
\begin{split}
	|f(u) - f(v)| &\leq L_0 |u-v|, \\ 
	|i(u) - i(v)| &\leq L_0 |u-v|.
\end{split}
\end{equation}

A possible choice for $\tff$ is so that the method defined by \eqref{eqn21z} is a Runge-Kutta method.  The following definition of a Runge-Kutta method has been modified from \cite[chap. II]{HLW} for the case of autonomous ODEs.

\begin{definition}
Let $b_i$, $a_{ij}$ $(i,j=1,\dotsc,s)$ be real numbers and let $h>0$ be the time step.  One step of an \emph{$s$-stage Runge-Kutta method} defining a map $x \mapsto x'$ for approximating the solution to \eqref{p1} is defined by 
\begin{align}
	k_i & = f\left( x + h \sum_{j=1}^s a_{ij} k_j \right) \qquad i=1,\dotsc,s, \label{eqn41z} \\
	x' &= x + h \sum_{i=1}^s b_i k_i. \label{eqn41zz}
\end{align}
\end{definition}

For $\tff$ in \eqref{eqn21z} to correspond to an $s$-stage Runge-Kutta method then we must define 
\begin{equation}
\label{eqn45z}
	\tff = \tff(x,h) := \sum_{i=1}^s b_i k_i, 
\end{equation}
where the $k_i$ are (implicitly) defined by \eqref{eqn41z}.  Even though the $k_i$ may be implicitly defined by \eqref{eqn41z}, as for Implicit Runge-Kutta methods, we may use the map $\cK : \bbR^d \times \bbR_+ \rightarrow \bbR^{d \times s}$ (defined below in Lemma \ref{lem rk exist}) to explicitly represent each $k_i$ in terms of $x$ and $h$ as the $i^{\rm th}$ column of $\cK(x,h)$.  Since the $k_i$ do not depend on $x'$ we get a $\tff$ that does not depend on $x'$ and we may write $\tff = \tff(x,h)$ instead of $\tff = \tff(x,x',h)$.

For a given $s$-stage Runge-Kutta method two constants that will be useful are $$
	A_1 := \max_i \sum_{j=1}^s |a_{ij}| \quad \mbox{and} \quad A_2 := \sum_{i=1}^s |b_i|.
$$

For completeness, the following lemma gives conditions on $h$ so that an $s$-stage Runge-Kutta method is well-defined in the sense that the map $x \mapsto \lbrace k_i : i=1,\dotsc, s \rbrace$ is locally well-defined (so that the map $x \mapsto x'$ is also locally well-defined).  Although it is a bespoke result for this paper, the proof is very similar to that given for \cite[Thm. II.7.2]{SolvingODEs1}.

\begin{lemma}
\label{lem rk exist}
Let $B \subset \bbR^d$ be a bounded set and let $R_0>0$ be a constant.  Let $L_0$ be the constant from \eqref{eqn40z}, and let $C_1$ be the constant from \eqref{eqn2}.  Define 
$$
	H := \min \lbrace \tsfrac{1}{L_0 A_1}, \tsfrac{1}{2A_1(C_1 + L_0/R_0)R_0} , \tsfrac{1}{2L_0 A_1 (C_1 + L_0/R_0)R_0} \rbrace.
$$  
Then for each $x \in B$, $h \in [0,H)$ and $u \in B_{2R_0}(x)$ there exists a unique $K = [k_1 k_2 \dotsi k_s] \in \lbrace [m_1 m_2 \dotsi m_s] \in \bbR^{d\times s} : |m_i - f(x) | \leq \tsfrac{L_0|i(x)|}{R_0} \rbrace$ satisfying 
\begin{equation}
\label{eqn45zz}
	k_i  = f\left( u + h \sum_{j=1}^s a_{ij} k_j \right) \qquad i=1,\dotsc,s.
\end{equation}
Let us define a map $\cK : \bbR^d \times \bbR_+ \rightarrow \bbR^{d \times s}$ such that $K = \cK(u,h)$.
\end{lemma}

\begin{proof}
We again apply Banach's Fixed Point Theorem (Theorem \ref{thm banach}).  Fix $x\in B$, $h \in [0,H)$ and $u \in B_{2R_0}(x)$.  For $K = [k_1 k_2 \dotsi k_s] \in \bbR^{d \times s}$ let $\nrm{K} = \max_i |k_i|$ and define $X :=\lbrace [m_1 m_2 \dotsi m_s] \in \bbR^{d\times s} : | m_i - f(x) | \leq \tsfrac{L_0|i(x)|}{R_0} \rbrace$ and $T: X \rightarrow \bbR^{d \times s}$ by 
$$
	T(K) := [l_1 l_2 \dotsi l_s] \qquad \mbox{where} \qquad l_i := f\left( u + h \sum_{j=1}^s a_{ij} k_j \right)
$$
for each $K = [k_1 k_2 \dotsi k_s] \in X$.  To apply Theorem \ref{thm banach} we must show that $T(K) \in X$ for all $K \in X$, and that $T : X \rightarrow X$ is a contraction.

Let $K = [k_1 k_2 \dotsi k_s] ,K' = [k_1' k_2' \dotsi k_s'] \in X$, $T(K) = [l_1 l_2 \dotsi l_s]$ and $T(K') = [l_1' l_2' \dotsi l_s']$.  First note that by the definition of $X$ and using \eqref{eqn2} we have 
\begin{equation}
\label{eqn46z}
	\nrm{K} \leq \max_i (|f(x)| + |k_i - f(x)|) \leq |f(x)| + \tsfrac{L_0|i(x)|}{R_0} \leq \left( C_1 + \tsfrac{L_0}{R_0} \right) |i(x)|.
\end{equation}
Using this, and since $h \leq \tsfrac{1}{2A_1(C_1+L_0/R_0)R_0}$ and $u \in B_{2R_0}(x)$, we have
\begin{equation}
\label{eqn47z}
	\left| u + h \sum_{j=1}^s a_{ij} k_j - x \right| 
	\leq |u-x| + h A_1 \nrm{K} 
	\leq \tsfrac{|i(x)|}{2R_0} + h A_1 \left(C_1 + \tsfrac{L_0}{R_0} \right) |i(x)|
	\leq \tsfrac{|i(x)|}{R_0}.
\end{equation}	
Hence $u + h \sum_{j=1}^s a_{ij} k_j \in B_{R_0}(x)$.  Using \eqref{eqn40z} and \eqref{eqn47z} we then get
$$
	|l_i - f(x)| 
	\leq L_0 \left| u + h \sum_{j=1}^s a_{ij} k_j - x \right| 
	\leq \tsfrac{L_0|i(x)|}{R_0}.
$$
Hence $T(K) \in X$.  By \eqref{eqn40z} we also have
$$
	|l_i - l_i'| \leq h L_0 \left| \sum_{j=1}^{s} a_{ij} (k_j - k_j') \right| \leq h L_0 A_1 \nrm{K-K'},
$$
so $\nrm{T(K) - T(K')} < h L_0 A_1 \nrm{K-K'}$.  Since $h < \tsfrac{1}{L_0 A_1}$, $T:X \rightarrow X$ is a contraction and the result then follows by applying Theorem \ref{thm banach}.
\end{proof}

The following lemma is a technical result for the subsequent corollary.

\begin{lemma}
\label{lem tff}
Let $\tff(x,h)$ be defined by \eqref{eqn45z}, where the $k_i$ are defined by an $s$-stage Runge-Kutta method with $\sum_{i=1}^s b_i = 1$.  Let $B \subset \bbR^d$ be a bounded set and $R_0 >0$ a constant.  Let $L_0$ be the constant from \eqref{eqn40z} and $C_1$ be the constant from \eqref{eqn2}.  

Then $\tff(x,0) = f(x)$ for every $x \in B$, and there exist positive constants $R$, $L$ and $H$ such that for each $x \in B$, and for all $u,v \in B_R(x)$ and $h\in [0,H)$
\begin{align*}
	|\tff(u,h) - \tff(v,h)| &\leq L |u-v|, \\
	|\tff(x,h) - \tff(x,0)| &\leq L h |i(x)|.
\end{align*}
\end{lemma}

\begin{proof}
Define $R := 2 R_0$, $L := L_0 A_2 \max \lbrace 2, A_1 (C_1 + \tsfrac{L_0}{R_0}) \rbrace$ and 
$$
	H := \min \lbrace \tsfrac{1}{2L_0 A_1} , \tsfrac{1}{2A_1(C_1 + L_0/R_0)R_0} , \tsfrac{1}{2L_0 A_1 (C_1 + L_0/R_0)R_0} \rbrace,
$$
and fix $x \in B$.  If $h = 0$, then $k_i = k_i(x,0) = f(x)$ for all $i$, and using $\sum_{i=1}^s b_i = 1$ we get $\tff(x,0) = \sum_{i=1}^s b_i f(x) = f(x)$.

For $h \in (0,H)$ and $u,v \in B_R(x)$, we have from Lemma \ref{lem rk exist} that there exist unique $K_u = \cK(u,h)$ and $K_v = \cK(v,h)$ in $\lbrace [m_1 m_2 \dotsi m_s] \in \bbR^{d\times s} : | m_i - f(x) | \leq \tsfrac{L_0|i(x)|}{R_0} \rbrace$ satisfying \eqref{eqn45zz}.  Using \eqref{eqn45zz}, \eqref{eqn40z} and $h \leq \tsfrac{1}{2L_0 A_1}$, we have
\begin{equation}
\begin{split}
	\nrm{K_u - K_v} 
	&\leq L_0 |u-v| + h L_0 A_1 \nrm{K_u - K_v} \\
	&\leq L_0 |u-v| + \tsfrac{1}{2} \nrm{K_u - K_v}.
\end{split}
\end{equation}
Hence, $\nrm{K_u - K_v} \leq 2 L_0 |u-v|$.  Using this, \eqref{eqn45z} and $A_2 := \sum_{i=1}^s |b_i|$ we get
$$
	|\tff(u,h) - \tff(v,h)| 
	\leq A_2 \nrm{K_u - K_v} \leq 2 L_0 A_2 |u-v| \leq L |u-v|.
$$
Finally, using \eqref{eqn45z}, $\sum_{i=1}^s b_i = 1$, \eqref{eqn40z}, \eqref{eqn46z} and writing $k_i(x,h)$ for the $i^{\rm th}$ column of $\cK(x,h)$,
\begin{align*}
	|\tff(x,h) - \tff(x,0)| 
	&= \left| \sum_{i=1}^s b_i k_i(x,h) - \sum_{i=1}^s b_i f(x) \right| \\ 
	&\leq A_2 \max_i |k_i(x,h) - f(x)| \\
	&= A_2 \max_i \left| f\left(x + h \sum_{j=1}^s a_{ij} k_j(x,h) \right)- f(x) \right| \\
	&\leq A_2 L_0 h A_1 \nrm{K} 
	\leq A_2 L_0 h A_1 \left( C_1 + \tsfrac{L_0}{R_0} \right) |i(x)| 
	\leq L h |i(x)|.
\end{align*}
\end{proof}

\begin{corollary}
All consistent Runge-Kutta methods (so that $\sum_{i=1}^s b_i = 1$) define a $\tff$ (see \eqref{eqn45z}) that satisfies condition \eqref{eqn11z} in Theorem \ref{thm exist}.
\end{corollary}

If we define $\tff$ corresponding to an explicit $s$-stage Runge-Kutta method (where $a_{ij} = 0$ for $i \leq j$), then the $k_i$ in \eqref{eqn41z} are defined explicitly in terms of $x$ and $h$ and $\tff(x,h)$ may be calculated explicitly (instead of using the map $\cK$ which may not be computed explicitly).  To obtain a Runge-Kutta method that is of order $p$ there are additional constraints on the $a_{ij}$ and $b_i$ (e.g. see \cite[p. 29]{HLW}).

To check whether or not $\bii = \bii(x,x')$ satisfies \eqref{eqn11z1} we only need to ensure that it is locally Lipschitz continuous since $\bii(x,x) = i(x)$ is already satisfied by the definition of a discrete gradient.  Moreover, since $i(x)$ is locally Lipschitz, it easily follows that both the coordinate increment discrete gradient (see \cite{abeitoh}) and the one used in the average vector field method (see e.g. \cite{MQR99,quispelmclaren08}) are also locally Lipschitz.  For general $I$ there are no known explicitly defined discrete gradients.  However, if $I$ is quadratic then $i(x)$ is linear and we may define $\bii(x,x')$ so that it is linear in $x'$ by taking 
\begin{equation}
\label{eqn48z}
	\bii(x,x') := i\left( \tsfrac{x + x'}{2} \right) = \tsfrac{1}{2}(i(x) + i(x')).
\end{equation}

There is considerable freedom over how we choose $\tii$, $\hii$ and $\brii$ in \eqref{eqn4}, and to apply Theorems \ref{thm exist} and \ref{thm error2} they only need to satisfy \eqref{eqn12z} and \eqref{eqn30z1}.  For example, we may define $\tii$ to be any of the following (and similarly for $\hii$ and $\brii$):
\begin{align*}
	\tii(x) &= i(x), \\
	\tii(x') &= i(x'), \\
	\tii(x,x') &= \tsfrac{1}{2} (i(x) + i(x')), \\
	\tii(x,x') &= i\left( \tsfrac{x + x'}{2} \right), \\
	\tii(x,x') &= \bii(x,x'), \\
	\tii(x,h) &= i(y) \qquad \mbox{where $y = x + h \tff(x,y,h)$}.
\end{align*}
Except for the final case when $\tii(x,h) = i(y)$, it is obvious that all of these choices for $\tii$ satisfy \eqref{eqn12z} since $i(x)$ is locally Lipschitz.  To confirm that $\tii(x,h) = i(y)$ satisfies \eqref{eqn12z} we prove the following two lemmas.

\begin{lemma}
\label{lem exist f2}
Let $B \subset \bbR^d$ be a bounded set and suppose there exist positive constants $R$, $L$ and $H$ such that for each $x \in B$, and for all $u,v,w \in B_{R}(x)$ and $h \in [0,H)$, $\tff : \bbR^d \times \bbR^d \times \bbR_+ \rightarrow \bbR^d$ satisfies \eqref{eqn11z}.  Let $C_1$ be the constant from \eqref{eqn2}.

Then for each $x \in B$, $u \in B_{2R}(x)$ and $h < \min \lbrace H, \tsfrac{1}{2R}, \tsfrac{1}{(C_1 + 2L/R) R}, \tsfrac{1}{L}\rbrace$ there exists a unique $y \in B_{R}(x)$ satisfying
$$
	y = u + h \tff(u,y,h).
$$
Moreover, if $u = x$ then $y \in B_{2R}(x)$.  
\end{lemma}

\begin{proof}
Fix $x \in B$, $u \in B_{2R}(x)$ and $h < \min \lbrace H, \tsfrac{1}{2R}, \tsfrac{1}{(C_1 + 2L/R) R}, \tsfrac{1}{L}\rbrace$.  Define $X := B_{R}(x)$ and $T: X \rightarrow \bbR^d$ by $T(z):= u + h \tff(u,z,h)$ for each $z \in X$.  Let $z,z' \in X$.  Using \eqref{eqn11z}, \eqref{eqn2}, $u \in B_{2R}(x)$, $z \in B_{R}(x)$ and $h \leq \lbrace \tsfrac{1}{2R}, \tsfrac{1}{(C_1 + 2L/R) R} \rbrace$ we get
\begin{align*}
	|T(z) - x| &\leq |u-x| + h |\tff(u,z,h)| \\
	&\leq |u-x| + h ( |f(x)| + L |u-x| + L |z-x| + L h |i(x)| ) \\
	&\leq \tsfrac{|i(x)|}{2R} + h \left( C_1 + \tsfrac{L}{2R} + \tsfrac{L}{R} + \tsfrac{L}{2R} \right) |i(x)| 
	\leq \tsfrac{|i(x)|}{R},
\end{align*}
so $T(z) \in X$.  From \eqref{eqn11z} we also get $
	|T(z) - T(z')| = h |\tff(u,z,h) - \tff(u,z',h)| \leq h L |z-z'|$.  Since $h < \tsfrac{1}{L}$, $T: X \rightarrow X$ is a contraction and the first part of the result then follows by applying Theorem \ref{thm banach}.

If $u=x$ then repeating this argument with $X := B_{2R}(x)$ yields $y \in B_{2R}(x)$.
\end{proof}

\begin{lemma}
\label{lem tii}
Let $B \subset \bbR^d$ be a bounded set and let $R_0$ be a positive constant.  Let $L_0$ be the constant from \eqref{eqn40z}, and let $C_1$ be the constant from \eqref{eqn2}.  Suppose there exist positive constants $R$, $L$ and $H$ such that for each $x \in B$, and all $u,v,w \in B_{R}(x)$ and $h \in [0,H)$ that $\tff : \bbR^d \times \bbR^d \times \bbR_+ \rightarrow \bbR^d$ satisfies \eqref{eqn11z}.  Define $R_1 := 2R$, $L_1 := \max \lbrace L, L_0, L_0 (C_1 + \tsfrac{L}{R} ) \rbrace$,  
$
	H_1 := \min \lbrace H, \tsfrac{1}{2R}, \tsfrac{1}{(C_1 + 2L/R) R}, \tsfrac{1}{2 L}\rbrace,
$
and $\tii : \bbR^d \times \bbR_+ \rightarrow \bbR^d$ such that 
$$
	\tii(u,h) := i(y) \qquad \mbox{where $y$ satisfies $y = u + h \tff(u,y,h)$}
$$
for all $x \in B$, $u \in B_{R_1}(x)$ and $h \in [0,H_1)$.

Then for each $x \in B$, and for all $u,v \in B_{R_1}(x)$ and $h \in [0,H_1)$, $\tii$ satisfies
\begin{align*}
	\tii(x,0) &= i(x), \\
	|\tii(u,h) - \tii(v,h)| &\leq L_1 |u-v|, \\
	|\tii(x,h) - \tii(x,0)| &\leq L_1 h |i(x)|.
\end{align*}
\end{lemma}

\begin{proof}
Fix $x \in B$, $u,v \in B_{R_1}(x)$ and $h \in [0,H_1)$.  Let $y = y(x,h)$ denote the solution to $y = x + h \tff(x,y,h)$, and similarly for $y(u,h)$ and $y(v,h)$.  From Lemma \ref{lem exist f2} we know these solutions exist and that $y(x,h) \in B_{2R}(x)$ and $y(u,h), y(v,h) \in B_{R}(x)$.

If $h = 0$ then $y = x$ and $\tii(x,0) = i(x)$.  

If $h \neq 0$, since $y(u,h),y(v,h) \in B_{R}(x)$, we may use \eqref{eqn11z} and $h \leq \tsfrac{1}{2L}$ to get
\begin{align*}
	|y(u,h) - y(v,h)| &= h |\tff(u,y(u,h),h) - \tff(v,y(v,h),h)| \\
	&\leq h L |u-v| + h L |y(u,h) - y(v,h)| \\
	&\leq \tsfrac{1}{2} |u-v| + \tsfrac{1}{2} |y(u,h) - y(v,h)|,
\end{align*}
and so $|y(u,h) - y(v,h)| \leq |u-v|$.  Using this and \eqref{eqn40z} we get
$$
	|\tii(u,h) - \tii(v,h)| 
	\!=\! |i(y(u,h)) - i(y(v,h))| 
	\! \leq \!  L_0 | y(u,h) - y(v,h) | 
	\! \leq \! L_0 | u-v| \! \leq \!  L_1 |u-v|.
$$
Using \eqref{eqn40z}, \eqref{eqn11z}, \eqref{eqn2}, $y(x,h) \in B_{2R}(x)$ and $h \leq \tsfrac{1}{2 R}$ we also get
\begin{align*}
	|\tii(x,h) - \tii(x,0)|
	&= |i(y(x,h)) - i(x)| \\
	&\leq L_0 |y(x,h) - x| 
	= L_0 h | \tff(x,y(x,h),h) | \\
	&\leq L_0 h (|f(x)| + L |y(x,h) - x| + L h |i(x)| ) \\
	&\leq L_0 h \left( C_1 + \tsfrac{L}{2R} + \tsfrac{L}{2R} \right) |i(x)| \\
	&\leq L_1 h |i(x)|.
\end{align*}
\end{proof}

In Theorem \ref{thm error2} we also require that $\tii$, $\bii$, $\hii$ and $\brii$ satisfy condition \eqref{eqn30z1}.  By taking $\hii(x,x',h) := \tii(x,x',h)$ and $\brii(x,x',h) := \bii(x,x')$ then this is achieved trivially, and the resulting method is equivalent to a projection method (see \cite{NMQSZ}).  Unfortunately, in this case the system of equations to solve at each time step is nonlinear in general.

A method that is almost a projection method is the following.  For general $\tii$, $\bii$ and $\tff$ satisfying the conditions in Theorem \ref{thm exist} and Theorem \ref{thm error2}, define $\hii(x,x',h) := \tii(x,x',h)$ and $\brii(x,h) := \bii(x,y(x,h))$ where $y=y(x,h)$ satisfies $y = x + h \tff(x,y,h)$.  The following two lemmas ensure that $\brii$ satisfies \eqref{eqn12z} and $\tii$, $\bii$, $\hii$ and $\brii$ satisfy \eqref{eqn30z1}.

\begin{lemma}
\label{lem brii}
Let $B \subset \bbR^d$ be a bounded set and let $C_1$ be the constant from \eqref{eqn2}.  Suppose there exist positive constants $R$, $L$ and $H$ such that for each $x \in B$, and all $u,v,w \in B_{R}(x)$ and $h \in [0,H)$ that $\tff : \bbR^d \times \bbR^d \times \bbR_+ \rightarrow \bbR^d$ satisfies \eqref{eqn11z}, and $\bii : \bbR^d \times \bbR^d \rightarrow \bbR^d$ is a discrete gradient of $I$ satisfying \eqref{eqn11z1}.  Define $R_1 := 2R$, $L_1 := \max \lbrace 2L, L (C_1 + \tsfrac{L}{R}) \rbrace$ and $H_1 := \min \lbrace H, \tsfrac{1}{2R}, \tsfrac{1}{L(C_1 + L/R)}, \tsfrac{1}{2L} \rbrace$, and also define $\brii : \bbR^d \times \bbR_+ \rightarrow \bbR^d$ such that
$$
	\brii(u,h) := \bii(u,y) \qquad \mbox{where $y$ satisfies $y = u + h \tff(u,y,h)$},
$$
for all $x \in B$, $u \in B_{R_1}(x)$ and $h \in [0,H_1)$.

Then for each $x \in B$, and for all $u,v \in B_{R_1}(x)$ and $h \in [0,H_1)$, $\brii$ satisfies
\begin{align*}
	\brii(x,0) &= i(x), \\
	|\brii(u,h) - \brii(v,h)| &\leq L_1 |u-v|, \\
	|\brii(x,h) - \brii(x,0)| &\leq L_1 h |i(x)|.
\end{align*}
\end{lemma}

The proof of this result is very similar to the proof of Lemma \ref{lem tii} so we omit it.

\begin{lemma}
\label{lem brii2}
Let $B \subset \bbR^d$ be a compact set and suppose that $R,L,H,\tff,\tii$ and $\bii$ all satisfy the conditions of Theorem \ref{thm exist}.  Define $\hii \equiv \tii$.  Define $R_1$, $L_1$, $H_1$ and $\brii$ as in Lemma \ref{lem brii}, and define
$$
	R_1' := \max \{ R_1, 10 L_1 \} \qquad \mbox{and} \qquad 
	H_1' := \min \lbrace H_1, \tsfrac{1}{10 L_1} , \tsfrac{1}{6C_2 R_1'},\tsfrac{1}{(36C_2 + 6) L_1}\rbrace.
$$

Then $\tff$, $\tii$, $\bii$, $\hii$, $\brii$ $R_1'$ and $H_1'$ satisfy the conditions in Theorem \ref{thm exist} with $R$, $L$, $H$, $R'$ and $H'$ replaced by $R_1$, $L_1$, $H_1$, $R_1'$ and $H_1'$ respectively.  

Moreover, for each $x \in B$ and $h \in[0, H_1')$ 
\begin{enumerate}
\item let $y \in B_{6R_1'}(x)$ be the unique solution to \eqref{eqn21z} (that exists by Lemma \ref{lem exist f}), 
\item let $x' \in B_{R_1'}(x)$ be the unique solution to \eqref{disc1} with $\wS$ defined by \eqref{eqn4} (that exists by Theorem \ref{thm exist}), and
\item let $x(\cdot)$ denote the exact solution to \eqref{p1} satisfying $x(t)$ for some $t \in \bbR_+$.
\end{enumerate}
Also suppose that 
\begin{enumerate}
\setcounter{enumi}{3}
\item $\tff$ is such that the method defined by \eqref{eqn21z} is of  order $p$ for some $p \in \mathbb{N}$, i.e. there exist constants $C_3$ and $H_3 < H_1'$ such that 
$$
	|y-x(t+h)| \leq C_3 h^{p+1}, \quad \mbox{for all $h \in [0,H_3]$ and all $x \in B$}.
$$
\end{enumerate}
If we define $C_4 := \tsfrac{6L_1}{5} \max \lbrace 1, C_3 \rbrace$, then $\tii$, $\bii$, $\hii$ and $\brii$ satisfy \eqref{eqn30z1}.
\end{lemma}

\begin{proof}
The fact that $\tff$, $\tii$, $\bii$, $\hii$, $\brii$ $R_1'$ and $H_1'$ satisfy the conditions in Theorem \ref{thm exist} with $R$, $L$, $H$, $R'$ and $H'$ replaced by $R_1$, $L_1$, $H_1$, $R_1'$ and $H_1'$ respectively, follows from $R_1 \geq R$, $L_1 \geq L$ and $H_1 \leq H$.

Using \eqref{eqn12z1} and \eqref{eqn11z1} (with $L$ replaced by $L_1$) we get
\begin{align*}
	|\hii(x,x',h) \!\cdot\! \brii(x,x',h) \!-\! \tii(x,x',h)\!\cdot\! \bii(x,x')|
	&= |\tii(x,x',h) \cdot [\bii(x,y) - \bii(x,x') ]| \\
	&\leq \tsfrac{6}{5} |i(x)| L_1 |x'-y| \\
	&\leq \tsfrac{6L_1}{5} |i(x)| (|x' \!-\! x(t+h)| + |y \!-\! x(t+h)|) \\
	&\leq \tsfrac{6L_1}{5} \left(  |x'-x(t+h)| + C_3 h^{p+1} \right) |i(x)| \\
	&\leq C_4 \left(  |x'-x(t+h)| + h^{p+1} \right) |i(x)|.
\end{align*}
\end{proof}

This lemma leads to the following obvious corollary of Theorem \ref{thm error2}.

\begin{corollary}
With the same hypotheses as Lemma \ref{lem brii2}, then the discrete gradient method defined by \eqref{disc1} and \eqref{eqn4} is of order $p$.
\end{corollary}

If $I$ is quadratic then using \eqref{eqn48z} to define $\bii$ and an explicit $s$-stage Runge-Kutta method to define $\tff$ we can construct a linearly implicit discrete gradient method.  The following corollary is a direct consequence of \eqref{eqn48z}, Lemmas \ref{lem tff}, \ref{lem exist f2}, \ref{lem brii} and \ref{lem brii2} and Theorems \ref{thm exist} and \ref{thm error2}.

\begin{corollary}
\label{cor 1}
Suppose $I$ is quadratic and let $\tff$ correspond to an explicit $s$-stage Runge-Kutta method of order $p$, for some $p \in \mathbb{N}$.  Then the discrete gradient method defined by 
$$
	x' = \begin{cases}
		x + \tsfrac{h}{2} \wS(x,h) i(x) + \tsfrac{h}{2} \wS(x,h) i(x') & \mbox{if $i(x) \neq 0$}, \\
		x & \mbox{if $i(x) = 0$},
		\end{cases}
$$
where
$$
	\wS(x,h) := \frac{ \tff(x,h) i(x)^T - i(x) \tff(x,h)^T}{i(x) \cdot i(x+ \frac{h}{2} \tff(x,h))}
$$
is linearly implicit in $x'$, locally well-defined (in the sense that for sufficiently small $h$ there exists a locally unique $x'$ at each time step) and of order $p$.
\end{corollary}

\section{Numerical examples}
\label{sec examples}

In this section we experiment with using the new linearly implicit (when $I$ is quadratic) discrete gradient method constructed in Corollary \ref{cor 1}.  To demonstrate the efficiency gain due to only needing to solve a linear system at each time step we will compare it with the standard projection method from \cite{HLW} on a problem with quadratic $I$.  

The new discrete gradient method we suggested in Corollary \ref{cor 1} for the case when $I$ is quadratic corresponds to defining $\tii(x,x',h) = \hii(x,x',h) := i(x)$, $\bii(x,x') = \tsfrac{1}{2} (i(x) + i(x'))$ and $\brii(x,x',h) := \bii(x,y)$ where $y$ satisfies $y = x + h \tff(x,h)$ and $\tff$ is defined by an explicit $s$-stage Runge-Kutta method.  With these choices for $\tff$, $\tii$, $\bii$, $\hii$ and $\brii$ then the discrete gradient method defined by \eqref{disc1} and \eqref{eqn4} becomes the one defined in Corollary \ref{cor 1}.  

In our experiments below we use the classical explicit $4^{\rm th}$ order Runge-Kutta (RK4) method to define $\tff$.  It is defined by the Butcher tableau (see e.g. \cite[p. 30]{HLW}):
$$
	\begin{array}{c|cccc}
		\star & & & & \\
		\star & \tsfrac{1}{2} & & & \\
		\star & 0 & \tsfrac{1}{2} & & \\
		\star & 0 & 0 & 1 & \\
		\hline
		& \tsfrac{1}{6} & \tsfrac{1}{3} & \tsfrac{1}{3} & \tsfrac{1}{6}
	\end{array}
$$
The entries denoted by $\star$ are not required because we are only considering autonomous ODEs.  

Since $I$ is quadratic, $i(x)$ is linear and there exists a matrix $M$ and a vector $b$ such that $i(x) = M x + b$ for all $x \in \bbR^d$.  In the case when $i(x) \neq 0$, to obtain $x'$ at each time step of the method in Corollary \ref{cor 1}, we must solve the linear system
\begin{equation}
\label{eqn88}
	(\idop - \tsfrac{h}{2} \wS(x,h) M) x' = (\idop + \tsfrac{h}{2} \wS(x,h) M) x + h \wS(x,h) b,
\end{equation}
where $\idop \in \bbR^d$ is the identity matrix.  Note that the cost of computing $\tff(x,h)$ at each time step is essentially the same as the cost for computing the RK4 method, so we already know that computing this new discrete gradient method will cost more than the RK4 method.

To compare this new linearly implicit discrete gradient method with another integral preserving method we also consider the standard projection method (see Algorithm IV.4.2 in \cite{HLW}) with RK4 as the underlying method.  The algorithm is:
\begin{align*}
	&\mbox{1. Compute $y$ by solving $y = h \tff(x,h)$ (where $\tff$ corresponds to the RK4 method)}, \\
	&\mbox{2. Compute $x'$ by solving $x' = y + \lambda i(y)$ such that $I(x') = I(x)$ for $x'$ and $\lambda \in \bbR$}.
\end{align*}
Actually, step $2$ above is what is suggested in equation (4.5) of \cite{HLW}, after Algorithm IV.4.2, as a more convenient nonlinear system to solve (by reducing the number of evaluations of $i(\cdot)$ required).  To solve the nonlinear system in step $2$ Hairer, Lubich and Wanner use the following simplified Newton iteration (see \cite[p. 111]{HLW} for details)
$$
	\lambda_0 = 0 \qquad \mbox{and} \qquad \lambda_{i+1} = \lambda_i - \frac{I(y + \lambda_i i(y)) - I(x)}{i(y) \cdot i(y)}.
$$
Once this iterative scheme has converged to $\lambda^*$ then $x'$ is computed using $x' = x + \lambda^* i(y)$.  

\subsection{Modified rigid body motion}

In the following example we will compare the performance of our new discrete gradient method with the RK4 method and another integral preserving method, the standard projection method (all described above).  We first demonstrate the benefits of preserving the integral by inspecting phase space plots for the RK4 method and our new discrete gradient method.  We will see that the integral preserving method does a much better job of following the trajectory of the exact solution.  To compare the errors we include all three methods.  We will see that the errors for all three methods are of similar size, that all three methods are of the same order, and that our new discrete gradient method is more efficient than the standard projection method.

The example we use for our computations is a modification to the equations for rigid body motion in three dimensions (see e.g. Example 1.7 in \cite[p. 99]{HLW}).  
For a parameter $\alpha \in \bbR$, the augmented equations of motion for a body with centre of mass at the origin are
\begin{equation}
\label{eqn50z}
 \frac{\dd}{\dd t} \left[ \threebyone{x_1}{x_2}{x_3} \right] = \left[ \threebythree{0}{-x_3}{x_2 - \alpha x_1^2}{x_3}{0}{-x_1}{-x_2 + \alpha x_1^2}{x_1}{0} \right]
 \left[ \threebyone{x_1/I_1}{x_2/I_2}{x_3/I_3} \right],
\end{equation}
where $I_1$, $I_2$ and $I_3$ are also parameters.
In the case when $\alpha = 0$ this system reduces to the equations for rigid body motion where the vector $x = (x_1,x_2,x_3)^T$ is the angular momentum in the body frame and the $I_i$ parameters are the principal moments of inertia.  Moreover, when $\alpha = 0$ there are two quadratic first integrals, but in the general case when $\alpha \neq 0$ then the only first integral is,
$$
	I(x) = \frac{1}{2} \left( \frac{x_1^2}{I_1} + \frac{x_2^2}{I_2} + \frac{x_3^2}{I_3} \right).
$$
Notice that \eqref{eqn50z} has the form of \eqref{p3}.

In our computations we have taken $I_1 = 2$, $I_2 = 1$, $I_3 = 2/3$, $\alpha = 1$ and we have used the initial condition $x_0 = ( \cos (1.1) , 0 , \sin(1.1) )^T$ at $t=0$.  Except for $\alpha$, these are the same values used in \cite{HLW}.

In Figure \ref{fig1} we see that phase space, projected onto the $x_1x_3$-plane, is more accurately represented when we compute the solution using our new discrete gradient method instead of the RK4 method.  Here we have used two different time steps ($h = 0.5$ and $h=\tsfrac{100}{92}$) and computed up to a final time of $t = 500$.  In the plots, the solid grey line is the exact solution and the black dots are the approximate solution at each time step using either RK4 or our new discrete gradient method.  For the larger step size of $h=\tsfrac{100}{92}$ the RK4 method appears to converge to equilibrium which is the wrong type of asymptotic behaviour.  For our new discrete gradient method, while the errors are clearly quite large for this larger step size, the solution appears to be circulating around a periodic orbit which is the correct asymptotic behaviour for this example.  Another possibility with the RK4 method (not observed in this example) is that the solution will blow up at some critical time (for example with $\alpha = 2$, $h = \tsfrac{10}{31}$ at $t \approx 100$).  This cannot happen for integral preserving methods such as our new discrete gradient method.

\begin{figure}
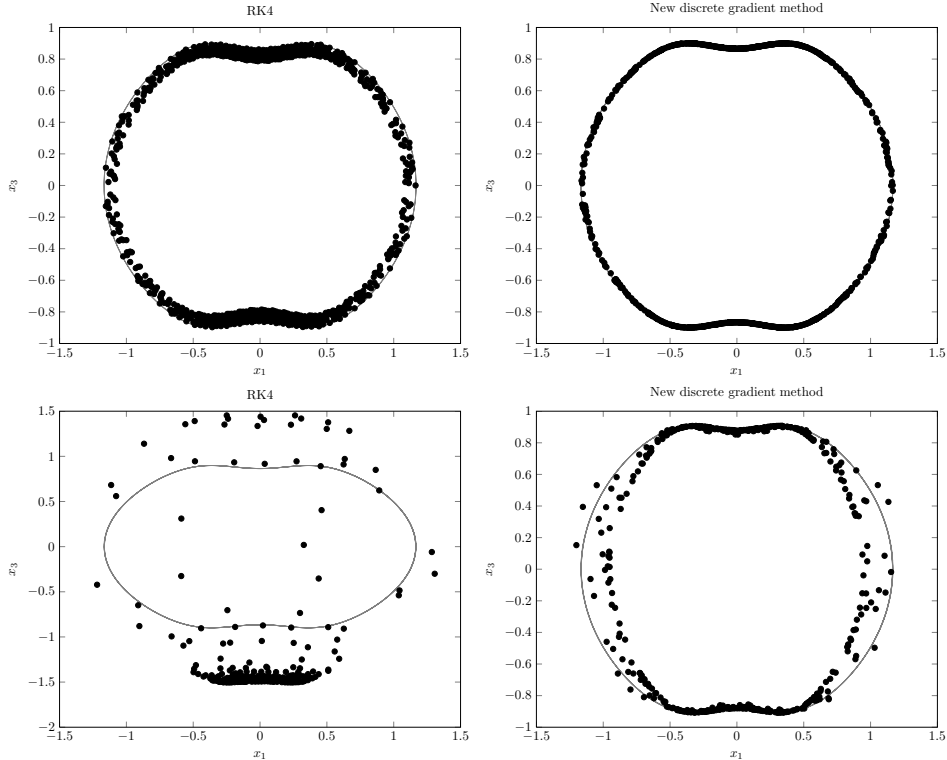
 
\begin{center}
\resizebox{0.49\textwidth}{!}{\input{fig1a.tikz}}
\resizebox{0.49\textwidth}{!}{\input{fig2a.tikz}}
\resizebox{0.49\textwidth}{!}{\input{fig1c.tikz}}
\resizebox{0.49\textwidth}{!}{\input{fig2c.tikz}}
\end{center}
\caption{Phase space projected onto the $x_1 x_3$-plane of modified rigid body motion comparing the performance of the RK4 method and our new discrete gradient method, computing the solution for $t \in [0,500]$ with $h = 0.5$ (top) and $h=\tsfrac{100}{92}$ (bottom).  The solid grey line is the exact solution and the black dots are the approximate solution at each time step.}
\label{fig1}
\end{figure}

In Figures \ref{fig2} and \ref{fig3} we compare the errors for the three different methods: RK4, the standard projection method, and our new discrete gradient method.  In Figure \ref{fig2} we have plotted the solution error and integral error versus time for the three different methods.  We see that the solution error is initially similar for all three methods, it grows as time increases, and then remains bounded.  The integral error plot clearly shows that the integral preserving methods preserve the integral up to double machine precision and are vastly superior in terms of preserving the integral than the non-integral preserving RK4 method.  These computations used a fixed time step of $h = 0.5$ for all three methods and computations were performed up to a final time of $t = 500$.

\begin{figure}
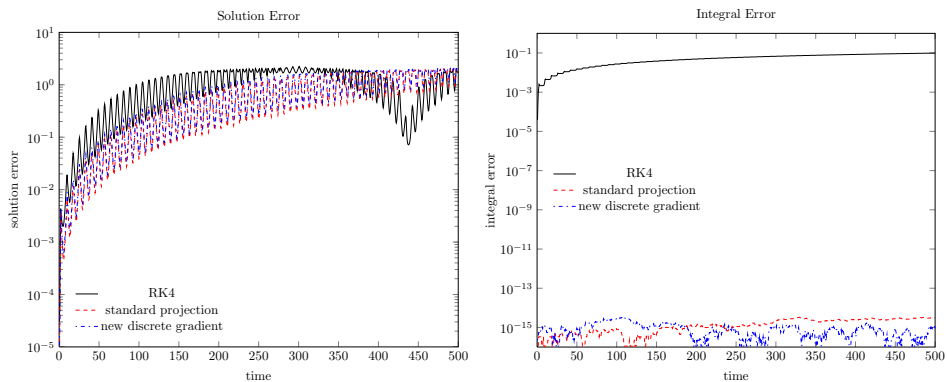

\begin{center}
\resizebox{0.49\textwidth}{!}{\input{fig3.tikz}}
\resizebox{0.49\textwidth}{!}{\input{fig4.tikz}}
\end{center}
\caption{Solution error and integral error vs. time for the RK4 method, the standard projection method and our new discrete gradient method. $h = 0.5$.}
\label{fig2}
\end{figure}

\begin{figure}
\begin{center}
\resizebox{0.49\textwidth}{!}{
%
%
%
%
\begin{tikzpicture}

\begin{loglogaxis}[%
view={0}{90},
width=0.7*6.01828521434821in,
height=0.7*4.74667979002625in,
scale only axis,
xmin=0.0001, xmax=1,
xminorticks=true,
xlabel={time step},
ymin=1e-14, ymax=100,
yminorticks=true,
ylabel={error at $t=100$},
title={Order},
legend style={at={(0.03,0.97)},anchor=north west,fill=none,draw=none,align=left}]
\addplot [
color=black,
solid,
mark=o,
mark options={solid}
]
coordinates{
 (0.000316227532033849,1.0516279854815e-13)(0.000774263481862878,3.48799540893117e-13)(0.0018957345971564,2.59676978897878e-12)(0.00464166357222429,8.65187577292017e-11)(0.0113649278327083,4.04938484965035e-09)(0.0278241513633834,2.289486035181e-07)(0.0681198910081744,1.55094213181425e-05)(0.1669449081803,0.00117404841215793)(0.408163265306122,0.0816504193263271)(1,1.16907223526726) 
};
\addlegendentry{RK4};

\addplot [
color=red,
dashed,
mark=square,
mark options={solid}
]
coordinates{
 (0.000316227532033849,3.35130378497843e-13)(0.000774263481862878,8.04014057064414e-13)(0.0018957345971564,1.9723147788463e-12)(0.00464166357222429,6.40902096008917e-11)(0.0113649278327083,2.41680124824533e-09)(0.0278241513633834,9.84731755949873e-08)(0.0681198910081744,4.55065188582967e-06)(0.1669449081803,0.00025076258281893)(0.408163265306122,0.0154214393806875)(1,1.86878072201194) 
};
\addlegendentry{standard projection};

\addplot [
color=blue,
dash pattern=on 1pt off 3pt on 3pt off 3pt,
mark=triangle,
mark options={solid}
]
coordinates{
 (0.000316227532033849,7.49259518793409e-13)(0.000774263481862878,3.9522338559326e-13)(0.0018957345971564,2.54646935622275e-12)(0.00464166357222429,6.44295464017134e-11)(0.0113649278327083,2.46156370741917e-09)(0.0278241513633834,1.02462249053849e-07)(0.0681198910081744,4.91092090673108e-06)(0.1669449081803,0.00028476265790702)(0.408163265306122,0.0187960291749281)(1,1.61836368300247) 
};
\addlegendentry{new discrete gradient};

\addplot [
color=black,
solid,
forget plot
]
coordinates{
 (0.0018957345971564,1.08425939327617e-08)(0.0113649278327083,1.08425939327617e-08) 
};
\addplot [
color=black,
solid,
forget plot
]
coordinates{
 (0.0018957345971564,8.39414773045786e-12)(0.0113649278327083,1.08425939327617e-08) 
};
\addplot [
color=black,
solid,
forget plot
]
coordinates{
 (0.0018957345971564,1.08425939327617e-08)(0.0018957345971564,8.39414773045786e-12) 
};
\node[above, inner sep=1mm, text=black]
at (axis cs:0.00464164700151258, 1.08425939327617e-08) {1};
\node[left, inner sep=1mm, text=black]
at (axis cs:0.0018957345971564, 3.01685822094721e-10) { 4};
\end{loglogaxis}
\end{tikzpicture}%
}
\resizebox{0.49\textwidth}{!}{
%
%
%
%
\begin{tikzpicture}

\begin{loglogaxis}[%
view={0}{90},
width=0.7*6.01828521434821in,
height=0.7*4.74667979002625in,
scale only axis,
xmin=0.001, xmax=100,
xminorticks=true,
xlabel={CPU time},
ymin=1e-14, ymax=100,
yminorticks=true,
ylabel={error at $t=100$},
title={Efficiency},
legend style={at={(0.03,0.03)},anchor=south west,fill=none,draw=none,align=left}]
\addplot [
color=black,
solid,
mark=o,
mark options={solid}
]
coordinates{
 (10.889196893,1.0516279854815e-13)(4.013374885,3.48799540893117e-13)(1.64340706,2.59676978897878e-12)(0.673498701,8.65187577292017e-11)(0.275061237,4.04938484965035e-09)(0.112873892,2.289486035181e-07)(0.047110364,1.55094213181425e-05)(0.019582631,0.00117404841215793)(0.008141649,0.0816504193263271)(0.003129434,1.16907223526726) 
};
\addlegendentry{RK4};

\addplot [
color=red,
dashed,
mark=square,
mark options={solid}
]
coordinates{
 (28.006395351,3.35130378497843e-13)(11.362693199,8.04014057064414e-13)(4.646532261,1.9723147788463e-12)(2.056465456,6.40902096008917e-11)(0.945223368,2.41680124824533e-09)(0.391724111,9.84731755949873e-08)(0.16433206,4.55065188582967e-06)(0.07099513,0.00025076258281893)(0.034619964,0.0154214393806875)(0.017796495,1.86878072201194) 
};
\addlegendentry{standard projection};

\addplot [
color=blue,
dash pattern=on 1pt off 3pt on 3pt off 3pt,
mark=triangle,
mark options={solid}
]
coordinates{
 (22.912716172,7.49259518793409e-13)(8.454516197,3.9522338559326e-13)(3.427211901,2.54646935622275e-12)(1.503371488,6.44295464017134e-11)(0.612830402,2.46156370741917e-09)(0.249561277,1.02462249053849e-07)(0.102636413,4.91092090673108e-06)(0.041422499,0.00028476265790702)(0.017614477,0.0187960291749281)(0.007125662,1.61836368300247) 
};
\addlegendentry{new discrete gradient};

\addplot [
color=black,
solid,
forget plot
]
coordinates{
 (0.391724111,1.79818125595171e-07)(2.056465456,1.79818125595171e-07) 
};
\addplot [
color=black,
solid,
forget plot
]
coordinates{
 (0.391724111,1.79818125595171e-07)(2.056465456,2.36738562305437e-10) 
};
\addplot [
color=black,
solid,
forget plot
]
coordinates{
 (2.056465456,1.79818125595171e-07)(2.056465456,2.36738562305437e-10) 
};
\node[above, inner sep=1mm, text=black]
at (axis cs:0.897533900503936, 1.79818125595171e-07) {1};
\node[right, inner sep=1mm, text=black]
at (axis cs:2.056465456, 6.52456010240224e-09) { 4};
\end{loglogaxis}
\end{tikzpicture}%
}
\end{center}
\caption{Order and efficiency of our new discrete gradient method compared with the RK4 method and the standard projection method.}
\label{fig3}
\end{figure}
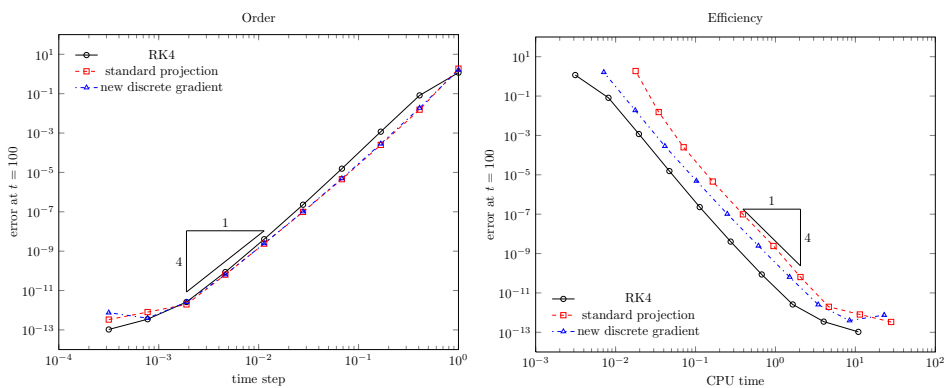

In Figure \ref{fig3} we compare the performance of the same three methods for different step sizes.  We are interested to see whether or not our new discrete gradient method is of the same order as RK4 (order 4), and to compare the efficiency of our new discrete gradient method with another integral preserving method, the standard projection method, where a nonlinear system of equations must be solved at each time step.  By plotting the solution error at time $t=100$ for different step sizes ($h \in [10^{-3.5}, 1]$) in the left plot of Figure \ref{fig3}, we confirm that our new discrete gradient method is of order $4$, the same as RK4 which is the underlying method defining $\tff$.  In the right plot of Figure \ref{fig3} for the same range of step sizes we have plotted the solution error at $t=100$ against the CPU time required to compute the solution up to $t=100$.  In this way we can compare the efficiency of these methods.  The plot clearly shows that our new discrete gradient method is more efficient than the standard projection method for this problem because it yields smaller errors using less computational effort.  The plot also shows that the RK4 method is more efficient again.  Since both integral preserving methods effectively compute the RK4 approximation within their methods the computational cost required by these methods is more than RK4.  Moreover, in the left plot of Figure \ref{fig3} we saw that the size of the error for all three methods is similar.  For these reasons RK4 is the most efficient method, however, RK4 does not preserve the integral and over longer time intervals it often has the wrong asymptotic behaviour.  

Also notice in Figure \ref{fig3} (right) that the difference in efficiency between our new discrete gradient method and the standard projection method is more pronounced for larger time steps.  This is probably due to the fact that the initial guess (the RK4 solution) in the Newton iteration for calculating the projection step in the standard projection method is more accurate for smaller time steps, resulting in fewer iterations until the convergence test is satisfied.

\subsection{A time step criterion}

A key feature of the existence and order of accuracy results in this paper is the fact that we may take $x$ as close as we like to a critical point of $I$ without any additional constraints on the time step.  By considering different $x$ values, and computing a single time step to get $x'$ for different time steps we can show that this feature of our results is illustrated in the modified rigid body motion example.  As criteria for a valid time step we consider the denominator of $\wS$ (which should be positive) and the condition number (ratio between the largest and smallest eigenvalues) of the matrix $\idop - \tsfrac{h}{2} \wS M$ (see \eqref{eqn88}).  The initial points we consider are $x = (R \cos(1.1), 0 , R \sin(1,1))^T$ for $R = 1$ (this is the initial condition used in our earlier simulations and is far away from a critical point of $I$), respectively $R=0.1$ and $R=0.01$ (which is near to the critical point $(0,0,0)^T$ of $I$).

In Figure \ref{fig4} we have plotted condition number of $\idop - \tsfrac{h}{2} \wS M$, the denominator of $\wS$ and the error after a single time step vs. time step, for different starting $x$ ($R=1$, $R=0.1$ and $R=0.01$).  We see that as the time step is increased there seem to be critical values where the condition number blows up, the denominator veers down to zero, and the error no longer behaves with the same asymptotic behaviour with respect to the time step.  We see that for $x$ close to the critical point the largest allowable time step actually increases.  This is consistent with our theory. 

\begin{figure}
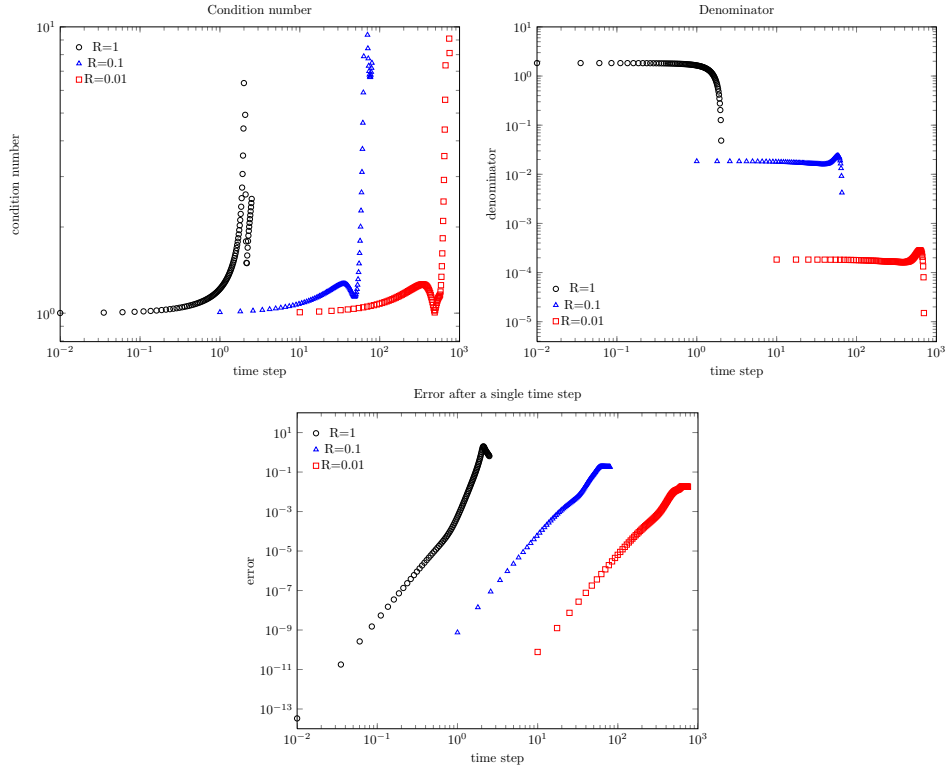

\begin{center}
\resizebox{0.49\textwidth}{!}{\input{fig8.tikz}}
\resizebox{0.49\textwidth}{!}{\input{fig8a.tikz}}
\resizebox{0.49\textwidth}{!}{\input{fig8b.tikz}}
\end{center}
\caption{Computations for a single time step of our new discrete gradient method applied to the modified rigid body motion example \eqref{eqn50z}, for different starting points which depend on $R$ (see text).  Top left: plot of condition number of $\idop - \tsfrac{h}{2} \wS M$ vs. time step; top right: plot of the value of the denominator of $\wS$ vs. time step; and bottom: plot of the error vs. time step.}
\label{fig4}
\end{figure}


\section{Conclusion}
\label{sec conclusion}

In this paper we have analysed discrete gradient methods from first principles.  We have established the bare essentials in terms of local Lipschitz continuity conditions and other criteria to ensure that these types of methods are locally well-defined and are of order $p$.  A key feature of our analysis is that we have removed any dependence of the time step on the distance to critical points of the preserved integral and all of the constants in our results are independent of $|i(x)|$.  

Although we have been careful to trace the value of constants through our proofs we do not make the claim that our constants are optimal.  The reasons for this are that we have assumed that the same constants $R$ and $L$ can be used in all of the inequalities in \eqref{eqn11z}, \eqref{eqn11z1} and \eqref{eqn12z}, and to simplify the presentation we sometimes used inequalities that were not completely sharp.  If we had more precise knowledge of the optimal constants for which \eqref{eqn2}, \eqref{eqn11z}, \eqref{eqn11z1}, \eqref{eqn12z}, \eqref{eqn30z} and \eqref{eqn30z1} hold, then we could repeat the arguments in the proofs of Theorems \ref{thm exist} and \ref{thm error2} to obtain better constants $R'$ and $H'$ in Theorem \ref{thm exist}, and $C_5$ and $H_5$ in Theorem \ref{thm error2}. 

As well as considering theoretical conditions for these methods we also developed results that will be useful for users of these methods for solving ODEs.  We have shown how Runge-Kutta methods can easily be used inside the framework of discrete gradient methods and we have also developed a new method that is linearly implicit when the integral to be preserved is quadratic, and of order $p$ for arbitrarily chosen $p \in \mathbb{N}$.  Our numerical experiments confirmed that, in this case, solving a linear system at each step instead of a nonlinear system led to significantly reduced computational cost.

The results in this paper can be easily applied to projection methods, see \cite{NMQSZ}, and further avenues for research include developing similar theory for discrete gradient methods applied to ODEs with Lyapunov functions, and discrete gradient methods applied to stiff ODEs, an issue not addressed here.

\section*{Acknowledgements}

This research was supported by the Australian Research Council.  Using the property of the discrete gradient in the bound of $T_1$ in \eqref{eqn21az} is a generalisation of an unpublished proof for projection methods by Ari Stern.


\end{document}